\documentclass[11pt]{amsart}

\usepackage[dvips]{graphicx}
\usepackage{pstricks}
\usepackage{array,amsmath,cases,amsfonts,amsthm}
\usepackage{lscape}
\usepackage{multirow}
\usepackage{pifont,bbding}
\usepackage{hyperref}

\newtheorem{proposition}{Proposition}
\newtheorem{conjecture}[proposition]{Conjecture}

\def\R{\mathbb{R}}
\def\distinct{\hspace{1ex} | \hspace{1ex}}
\def\maybe{\hspace{1ex} : \hspace{1ex}}

\begin{document}

\title{An Atlas of Legendrian Knots}
\author[W.\ Chongchitmate]{Wutichai Chongchitmate}
\address{Mathematics Department, University of California at Los
  Angeles, Los Angeles, CA 90095} 
\author[L.\ Ng]{Lenhard Ng}
\address{Mathematics Department, Duke University, Durham, NC 27708}
\urladdr{\href{http://www.math.duke.edu/~ng/}{http://www.math.duke.edu/\~{}ng/}}

\begin{abstract}
We present an atlas of Legendrian knots in standard contact
three-space. This gives a conjectural
Legendrian classification for all knots with arc index at most $9$, including
alternating knots through $7$ crossings and nonalternating knots
through $9$ crossings. Our method involves a computer search of grid
diagrams and applies to transverse knots as well. The
atlas incorporates a number of new, small examples of phenomena such
as transverse nonsimplicity and non-maximal non-destabilizable Legendrian
knots, and gives rise to new infinite families of transversely
nonsimple knots.
\end{abstract}

\maketitle

\section{Introduction}

A central problem in contact knot theory is the Legendrian and transverse classification problem: how to classify
all Legendrian and transverse knots of a particular topological type
in some contact $3$-manifold. This is an interesting question even for
the most basic case, $\R^3$ with the standard contact structure
$\ker(dz-y\,dx)$. Legendrian and transverse knots have been classified
in this case for a few families of knots, including the unknot
\cite{EF}, torus knots \cite{EH}, and twist knots \cite{ENV}. The
classification problem for most other knots, however, including many
``small'' knots, is currently wide open.

In this paper, we present a conjectural Legendrian and transverse
classification for all prime knots in $\R^3$ with arc index at most
$9$. This includes 
all prime knots with $7$ or fewer crossings, all prime nonalternating knots with
$9$ or fewer crossings, and an assortment of other nonalternating
knots. (One can use the prime classification to similarly classify
composite knots, by the results of \cite{EHconnectsum}.)
The classification is presented at the back end of this paper
in the form of a ``Legendrian knot atlas''. A corresponding atlas of
transverse knots can be deduced from this.

The strategy behind our atlas is a ``probabilistic'' approach to
enumerating Legendrian and transverse knots, based on expressing them
as grid diagrams. Two grid diagrams representing Legendrian or
transverse knots are isotopic if and only if they are related by a
sequence of elementary moves, some subset of the so-called Cromwell
moves. Roughly speaking, our algorithm enumerates all grid diagrams of
a particular size and attempts to determine which of them are related
by these Cromwell moves. Unfortunately, one of the Cromwell moves,
stabilization, changes the size of the grid diagram, and so the
algorithm cannot prove, in finite time, that two grid diagrams
represent nonisotopic Legendrian or transverse knots. Nevertheless we
can guess with some degree of confidence when two grids are isotopic,
under the assumption that if two grid diagrams of a certain size are
related by Cromwell moves, then they are related by moves that do not
increase the grid size by too much.

The result of the algorithm is a computer program that can show that
various grid diagrams are isotopic, and guesses that other grid
diagrams are not isotopic. This technique seems to be surprisingly
effective in classifying Legendrian and transverse knots. In many
cases, one can prove by hand, using various recently developed
invariants, that the isotopy classes of grid diagrams produced by the
program are indeed distinct.

We hope that the wealth of examples produced by the atlas will be
useful to researchers working in contact geometry and related
fields. A precursor of sorts to this atlas, an enumeration of
Legendrian representatives of knots through $9$ crossings by Melvin
and Shrestha \cite{MS}, has provided testing material for various
projects in contact topology, and many of the Melvin--Shrestha
examples appear in some guise as part of our atlas.

In compiling the atlas, we discovered examples of several interesting
phenomena for Legendrian and transverse knots that either had not been
seen before, or had only been seen in much more complicated
examples. In particular, the atlas contains:
\begin{itemize}
\item
Legendrian (respectively
transverse) knots that do
not maximize Thurston--Bennequin number (self-linking number) but are
not destabilizable;
\item
knots that can be proven to be transversely
nonsimple by inspection and a bit of knot Floer homology, without
computer verification or more complicated techniques;
\item
knots that can
be proven to be transversely nonsimple only through a recently
developed invariant, transverse homology, and not by knot Floer
homology;
\item
transverse knots that are conjecturally distinct from their transverse
mirrors;
\item
Legendrian knots with more than one linearized contact homology.
\end{itemize}

Indeed, some examples in the atlas can readily be generalized to give,
for instance, infinite families of knots that can be proven to be
transversely nonsimple by inspection. Furthermore, we show the following
result, with an analogous statement also holding for Legendrian knots.

\begin{proposition}
There are non-destabilizable prime transverse knots whose
self-linking number is arbitrarily far from maximal.
\end{proposition}

\noindent
We note that similar results have been obtained by Etnyre, LaFountain,
and Tosun, but with a completely different set of examples (cables of
torus knots).

There are a fair number of knots in the atlas (drawn in red) that we
conjecture, but
are currently unable to prove, are distinct. It would be interesting to
know if various ``modern'' techniques could be applied to these knots:
Massey products on
linearized contact homology \cite{CEKSW}, Legendrian Symplectic Field
Theory \cite{NgSFT}, and so forth.

The atlas itself is available as a standalone file from
\begin{center}
\href{http://www.math.duke.edu/~ng/atlas/}{\texttt{http://www.math.duke.edu/\~{}ng/atlas/}} 
\end{center}
where an analogous
atlas for unoriented two-component Legendrian links, as well as
various source files, can also be downloaded. Any future updates to the
atlas will be posted there as well.

In Section~\ref{sec:background}, we provide a quick summary of the
terms we use in the atlas, and describe the algorithm used to produce
the Legendrian knot atlas. We discuss the particular examples and families
of examples, illustrating the aforementioned unusual phenomena and
others, in Section~\ref{sec:phenomena}. The Legendrian knot atlas
itself comprises Section~\ref{sec:atlas}.

\section*{Acknowledgments}

The authors would like to thank John Etnyre, Hiroshi Matsuda, 
Dan Rutherford, Josh
Sabloff, and Shea Vela-Vick for illuminating discussions. Much of this
work appeared in the first author's undergraduate honors thesis at
Duke University, with support from the PRUV program at Duke. The
second author was supported by NSF grant DMS-0706777 and NSF CAREER
grant DMS-0846346.

\section{Background and Methodology}
\label{sec:background}

\subsection{Background}

Of the various ways to depict Legendrian and transverse knots in
standard contact $\R^3$, we will exclusively use grid diagrams. Here we
briefly recall the salient features of grid diagrams and their
relationship to Legendrian and transverse knots; a more detailed
discussion can be found in, e.g., \cite{NT,OST}, and a more general
introduction to contact knot theory in \cite{Etnyresurvey}.

A \textit{grid diagram} is an $n\times n$ square grid containing $n$
X's and $n$ O's in distinct squares, such that each row and each
column contains exactly one X and one O. The \textit{grid number} of a
grid diagram is $n$. Given a grid diagram, one can
obtain a diagram of an oriented link in $\R^3$ by connecting O's to
X's horizontally and X's to O's vertically, and having all vertical
line segments pass over all horizontal line segments wherever they
cross. We will use grid diagrams and the associated link diagrams
interchangeably. One can also obtain a front diagram for an oriented Legendrian
link in $\R^3$ by rotating the link diagram $45^\circ$
counterclockwise and smoothing corners. Any topological knot, and
indeed any Legendrian knot, can be represented by a grid diagram;
the \textit{arc index} of a topological knot is the minimum grid
number over all grid diagrams representing the knot.

\begin{figure}
\centerline{
\includegraphics[height=2.5in]{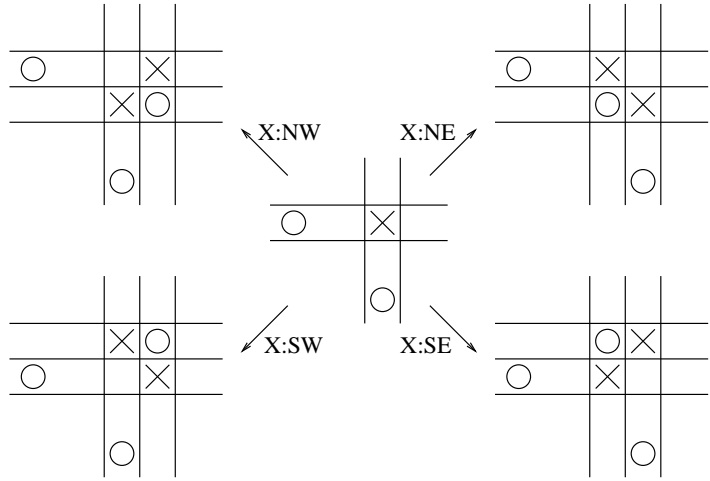}
}
\caption{The four types of X stabilizations on a grid diagram.
}
\label{fig:stabilization}
\end{figure}

There are three \textit{Cromwell moves} relating grid diagrams, each
of which preserves topological link type:
\begin{itemize}
\item
\textit{torus translation}, which moves the topmost row (or bottommost
row, leftmost column, rightmost column) of a grid
diagram to the bottommost row (topmost row, rightmost column, leftmost column) of the grid;
\item
\textit{commutation}, which switches adjacent rows (columns) in which
the segments connecting O's and X's are either disjoint or nested when
projected to a single horizontal (vertical) line;
\item
\textit{stabilization}, which increases grid number by $1$ and
replaces a single X (O) in the diagram by a $2\times 2$ square with
two X's (O's) and one O (X).
\end{itemize}
Of these, the most interesting to us is stabilization and its inverse
operation, destabilization. Stabilization comes in
eight flavors, four X stabilizations and four O stabilizations,
depending on whether a single X or O is replaced and the form of the
resulting $2\times 2$ square. It suffices for our purposes to consider
only the X stabilizations, which are depicted in
Figure~\ref{fig:stabilization} (the O stabilizations are redundant).

Two grid diagrams represent isotopic topological links if and only if
they are related by some sequence of Cromwell moves.
We can also consider Legendrian and transverse links, up to Legendrian
and transverse isotopy, to be grid diagrams modulo certain Cromwell moves:
\begin{itemize}
\item
\textit{Legendrian links} are grid diagrams modulo torus translation,
commutation, and X:NE and X:SW stabilization and destabilization;
\item
\textit{transverse links} are grid diagrams modulo torus translation,
commutation, and X:NE, X:SW, and X:SE stabilization and destabilization.
\end{itemize}
In this language, an enumeration of Legendrian or transverse links up
to isotopy becomes an enumeration of grid diagrams up to the
appropriate equivalence relation. Also, any Legendrian link can be
viewed as a transverse link; in contact topology, the resulting
transverse link is called the \textit{positive transverse pushoff} of
the Legendrian link.

The \textit{classical invariants} of Legendrian and transverse links
in standard contact $\R^3$, which are unchanged by Legendrian or
transverse isotopy, are defined in terms of a grid diagram as
follows:
\begin{itemize}
\item
the
\textit{Thurston--Bennequin number}, $tb$, is the writhe of the link
diagram (the number of crossings counted with sign) minus the number
of NE corners of the link diagram;
\item
the \textit{rotation number}, $r$, is $1/2$ of the total number of NE
and SW corners, counted with sign, where a corner is counted
positively if it is traversed down and to the right, and negatively if
it is traversed up and to the left;
\item
the \textit{self-linking number}, $sl$, is $tb-r$.
\end{itemize}
The Thurston--Bennequin and rotation numbers are invariant under
Legendrian isotopy, while the self-linking number is invariant under
transverse isotopy.

A topological knot type is \textit{Legendrian simple} (respectively
\textit{transversely simple}) if any two Legendrian (transverse) knots
of that type with the same $tb$ and $r$ ($sl$) are necessarily
Legendrian (transversely) isotopic. Any Legendrian simple knot is also
transversely simple. Proofs that various knot types are Legendrian
nonsimple have been obtained as applications of certain
``non-classical'' Legendrian invariants, such as the Legendrian
contact homology of Chekanov \cite{Ch} and Eliashberg and the ruling invariants
of Chekanov--Pushkar \cite{ChP} and Fuchs \cite{F}. It has historically been more
difficult to establish transverse nonsimplicity than Legendrian
nonsimplicity.

The operations of X:NW and X:SE stabilization descend to Legendrian
knots, where they become \textit{positive} and \textit{negative
  Legendrian stabilization} and change $(tb,r)$ by
$(-1,1)$ and $(-1,-1)$, respectively. The operations of positive and
negative Legendrian stabilization, which we denote by $S_+$ and $S_-$,
commute up to Legendrian isotopy: $S_+S_-(L) = S_-S_+(L)$.
Transverse knots can be
seen as Legendrian knots modulo negative Legendrian stabilization, and
the operation of X:NW stabilization descends to transverse knots,
where it is called \textit{transverse stabilization} and decreases
$sl$ by $2$. If a Legendrian or transverse knot is a stabilization
of another, then we say that it is \textit{destabilizable}. Since
destabilization increases $tb$ by $1$ (Legendrian) and $sl$ by $2$
(transverse), and $tb$ and $sl$ are bounded above in any given
topological type by a classical result of Bennequin, there are
non-destabilizable Legendrian and transverse knots in every knot type.

\begin{figure}
\centerline{
\raisebox{-0.2in}{\includegraphics[height=1in]{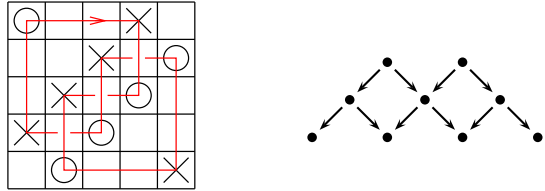}}
\hspace{0.5in}
\begin{pspicture}(3,1.4)
				\psdot(1,1.2)
				\psdot(2,1.2)
				\psline{->}(0.9,1.1)(0.6,0.8)
				\psline{->}(1.1,1.1)(1.4,0.8)
				\psline{->}(1.9,1.1)(1.6,0.8)
				\psline{->}(2.1,1.1)(2.4,0.8)
				\psdot(0.5,0.7)
				\psdot(1.5,0.7)
				\psdot(2.5,0.7)
				\psline{->}(0.4,0.6)(0.1,0.3)
				\psline{->}(0.6,0.6)(0.9,0.3)
				\psline{->}(1.4,0.6)(1.1,0.3)
				\psline{->}(1.6,0.6)(1.9,0.3)
				\psline{->}(2.4,0.6)(2.1,0.3)
				\psline{->}(2.6,0.6)(2.9,0.3)
				\psdot(0,0.2)
				\psdot(1,0.2)
				\psdot(2,0.2)
				\psdot(3,0.2)
			\end{pspicture}
}
\caption{
Grid diagram for a Legendrian left-handed trefoil with $(tb,r) =
(-6,1)$, and the Legendrian mountain range for the left-handed
trefoil. The top row of the mountain range depicts $(tb,r) = (-6,-1)$
and $(-6,1)$; the next row, $(tb,r) = (-7,-2)$, $(-7,0)$, and
$(-7,2)$; and so forth. The arrows represent positive (pointing to the
right) and negative (to the left) Legendrian stabilization. The
mountain range continues infinitely downwards by application of stabilizations.
Note that the
left-handed trefoil is Legendrian simple: each value of $(tb,r)$ has
at most one Legendrian representative.
}
\label{fig:mtnrange}
\end{figure}

In the atlas, we depict Legendrian knots of a particular topological
type via a \textit{Legendrian mountain range}, whereby isotopy classes
of Legendrian knots are plotted according to their $(tb,r)$, with $tb$
in the vertical direction and $r$ in the horizontal
direction. Positive/negative stabilization are depicted in the
mountain range by arrows pointing down and to the right/left. The
classification of transverse knots of a particular topological type
can be deduced from the Legendrian classification by modding out by
the effect of negative Legendrian stabilization. See
Figure~\ref{fig:mtnrange}.

Finally, we describe certain symmetries of Legendrian and transverse
knots that are useful to consider in the atlas. Given a Legendrian knot $L$,
one can reverse orientation to
obtain the Legendrian knot $-L$; this corresponds to switching X's and
O's in a grid diagram, and replaces $(tb,r)$ by $(tb,-r)$. (Since this
operation changes the underlying topological knot to its orientation
reverse, it may change topological knot type in general, but all of
the knots in the atlas are isotopic to their orientation reverses.) We remark that orientation reversal intertwines stabilizations: $S_+(-L) = -S_-(L)$.
One can also define the \textit{Legendrian mirror} $\mu(L)$ to be the
result of applying the contactomorphism $(x,y,z) \mapsto (-x,y,-z)$ to $L$;
this corresponds to rotating the grid diagram $180^\circ$, and
replaces $(tb,r)$ by $(tb,-r)$. The combination of the two symmetries,
$L \mapsto -\mu(L)$, descends to transverse knots and is called the
\textit{transverse mirror}.

\subsection{Methodology}

Here we give a brief, and somewhat simplified,
summary of the algorithm used to produce the
Legendrian knot atlas. More details can be found in \cite{Chong}; see
\href{http://www.math.duke.edu/~ng/atlas/}{\texttt{http://www.math.duke.edu/\~{}ng/atlas/}} for source files.

View the set of all grid diagrams (of arbitrary size), modulo
torus translation, as an infinite graph $\Gamma$. Connect two vertices of $\Gamma$ if they are related by a single commutation move, or a single
stabilization move (of appropriate restricted type corresponding to
Legendrian or transverse isotopy). The connected components of $\Gamma$ are precisely isotopy classes of Legendrian or
transverse knots.

The graph $\Gamma$ has an increasing filtration of finite subgraphs $\Gamma_n$ whose vertices consist of all grid diagrams of size at most $n$. Determining connected components of $\Gamma_n$ is a case of the familiar pathfinding problem in computer science, and approximates the problem of determining connected components of $\Gamma$.

The algorithm first produces a list of vertices of $\Gamma_9$, using a
straightforward modification of the technique used by Jin, Kim, and
Lee \cite{JKL} to enumerate all prime knots with arc index at most
$10$. We eliminate grid diagrams corresponding to multicomponent links
and divide the rest according to their topological knot type and
classical invariants ($(tb,r)$ for Legendrian, $sl$ for
transverse). Given the remaining grid diagrams of a particular knot
type and classical invariants, the algorithm then runs a bidirectional
search to determine which diagrams are connected to each other in
$\Gamma_n$, where $n$ can be adjusted and depends on the knot type,
but is typically $10$ or $11$. This
allows us to reduce the set of grid diagrams to a smaller set that is
guessed by the program to represent pairwise nonisotopic Legendrian or
transverse knots. 

Several timesaving features have been incorporated into the actual
program, which is implemented in Java, including: eliminating grid
diagrams that include an adjacent X-O pair and are immediately
destabilizable; first considering unoriented grid diagrams (where X's
and O's are interchangeable); and 
adding edges corresponding to other moves that preserve Legendrian
isotopy type and grid size, including the $S_2$ move from
\cite{NT}. See \cite{Chong} for details. 

Because of our algorithm for constructing Legendrian knots, the
completeness of our table is related to the following. 

\begin{conjecture}
Any Legendrian knot of maximal Thurston--Bennequin number has a grid
diagram representative of minimal grid number. More generally, for a
topological knot $K$, let $\overline{tb}(K)$ and $\alpha(K)$ denote
the maximal Thurston--Bennequin number and arc index of $K$,
respectively; then any Legendrian knot of type $K$ and
Thurston--Bennequin number $\overline{tb}(K)-m$ can be represented by
a grid diagram of size $\alpha(K)+m$. 
\end{conjecture}

\noindent
We have expressed this statement, which is related to a question in \cite{Ngarc}, as a conjecture, although we suspect that it is probably false in general. However, it appears to be true for small knots---the program failed to find any counterexamples for small grid number---and the completeness of the atlas relies on the conjecture being true, or approximately true, for the knots in the table.

On a related note, it is interesting to find grid diagrams that are not minimal within their topological type, but nevertheless cannot be destabilized without first being stabilized; that is, non-minimal grid diagrams where torus translation and commutation (the Cromwell moves that preserve grid number) do not suffice to produce a destabilizable diagram. The above conjecture suggests that such grid diagrams may correspond to Legendrian knots that have non-maximal $tb$ but are non-destabilizable. It is easy to modify our program to find all such grid diagrams of a certain size. Indeed, the relevant grid diagrams of size at most $10$ (which then represent knots of arc index at most $9$) all produce non-maximal Legendrian knots that are either provably or conjecturally non-destabilizable. These appear in the atlas as non-maximal Legendrian knots of type $m(10_{139})$, $m(10_{145})$, $10_{161}$, $m(10_{161})$, $m(12n_{242})$, and $12n_{591}$. See also the discussion in Sections~\ref{ssec:transII} and~\ref{ssec:Legnondestab}.

\section{Notable Phenomena}
\label{sec:phenomena}

In this section, we observe instances of interesting behavior in the atlas. These include transverse nonsimplicity, which we extend to families beyond the knots in the atlas, and non-destabilizability for certain Legendrian knots. We also document the methods we use to distinguish various knots in the atlas.

\subsection{Transverse nonsimplicity I}
\label{ssec:transI}

Among knots with arc index at most $9$, the computer program guesses
that exactly $13$ are transversely nonsimple:
\begin{gather*}
m(7_2),~ m(7_6),~ 9_{44},~ m(9_{45}),~ 9_{48},~ 10_{128},~
m(10_{132}),\\
 10_{136},~ m(10_{140}),~ m(10_{145}),~ 10_{160},~
m(10_{161}),~ 12n_{591}.
\end{gather*}
These knots can be seen in the atlas as the ones whose mountain ranges
(conjecturally) contain distinct Legendrian knots with the same
$(tb,r)$ that remain distinct under repeated stabilization of one type
or the other. (These are pictorially represented in the atlas by
mountain ranges with boxes that persist under stabilization.)

It should be emphasized that the computer program cannot prove either
transverse simplicity or transverse nonsimplicity, but it can make predictions. Of the knots in
the table, the program guesses that $69$ are transversely simple. Of
these, $18$ are currently known to be transversely simple, precisely
corresponding to torus knots \cite{EH} and certain
twist knots \cite{EH,ENV}:
\begin{gather*}
3_1,~ m(3_1),~ 4_1,~ 5_1,~ m(5_1),~ 5_2,~ m(5_2),~ 6_1,~ m(6_1),~
7_1,~ m(7_1),~ 7_2,~ \\
8_{19},~ m(8_{19}),~ 10_{124},~ m(10_{124}),~
15n_{41185},~ m(15n_{41185}).
\end{gather*}
Proving transverse simplicity for the remaining $51$ knots appears to
be difficult and might involve convex surface techniques as in
\cite{EH,ENV}.

On the other hand, proving transverse nonsimplicity can sometimes be a
simple matter of applying one of the known transverse invariants.\footnote{The
techniques of Birman and Menasco \cite{BM2,BM} are another approach to
transverse nonsimplicity, but the knots of braid index $3$ that they
have proven to be transversely nonsimple all have arc index at least $10$ and are not covered in the atlas.
}
The $\widehat{\theta}$ transverse invariant in knot Floer homology of
Ozsv\'ath, Szab\'o, and Thurston \cite{OST} proves that $5$ of the
$13$ transversely nonsimple candidates listed above are
indeed transversely nonsimple: $m(10_{132})$, $m(10_{140})$,
$m(10_{145})$, $m(10_{161})$, and $12n_{591}$. In each of these cases,
the computer program of \cite{NOT} demonstrates that
$\widehat{\theta}$ is zero for one of the transverse representatives
and nonzero for the other. (In particular, this precise computation is
presented in \cite{NOT} for $m(10_{132})$ and $m(10_{140})$.)
A related transverse
invariant in knot Floer homology due to \cite{LOSS} has been used in
\cite{OSt} to
prove transverse nonsimplicity for $m(7_2)$.

Recently a new transverse invariant, transverse homology, has been
introduced by the second author in collaboration with Ekholm, Etnyre,
and Sullivan; see \cite{EENStransverse,Ngtranshom}. As described in
\cite{Ngtranshom}, transverse homology proves transverse nonsimplicity
for $10$ of the above $13$ candidates, including the $5$ also given
by knot Floer homology:
$m(7_2),$ $m(7_6),$ $9_{44}$, $9_{48}$,
$m(10_{132})$, $10_{136}$, $m(10_{140})$, $m(10_{145})$,
$m(10_{161})$, $12n_{591}$.

In all cases involving transverse nonsimplicity, the precise statement
is as follows: for a particular knot type, there are two grid diagrams
in the atlas representing Legendrian knots $L_1,L_2$ of that knot
type, such that the positive transverse pushoffs of $L_1,L_2$ are not
transversely isotopic. It follows that arbitrary negative
Legendrian stabilizations of $L_1,L_2$ are distinct, as are arbitrary
positive Legendrian stabilizations of $-L_1,-L_2$ (alternatively,
positive Legendrian stabilizations of $\mu(L_1),\mu(L_2)$). For an
enumeration of which specific grid diagrams in the atlas correspond to
distinct transverse knots, see Table~\ref{table:stab} at the end of the atlas.

There are $4$ instances where the atlas guesses, but the above
invariants so far fail to prove, that certain
transverse knots are distinct. These are the $3$ knots
$m(9_{45})$, $10_{128}$, and $10_{160}$, which we conjecture but
cannot prove are transversely nonsimple, and the knot $9_{44}$, where
the program finds three possibly distinct transverse knots but the
invariants can only distinguish two. In all these cases, the issue is
a subtle involutive operation on transverse knots called the
transverse mirror \cite{NT}. In terms of Legendrian knots, this
operation can be described as follows: given a Legendrian knot $L$, the positive
transverse pushoffs of $L$ and $-\mu(L)$ are defined to be transverse
mirrors. (In general, transverse mirrors are topologically related by
orientation reversal, but all of the topological knots in the atlas
are invariant under orientation reversal.)
Transverse mirrors are difficult to distinguish using the known
invariants, and the transverse mirror pairs in the $4$ knot types
above are conjectured but not proven to be distinct.



\subsection{Transverse nonsimplicity II}
\label{ssec:transII}

Three of the transversely nonsimple knot types described in the
previous section---$m(10_{145})$, $m(10_{161})$, and
$12n_{591}$---merit further discussion. These are knots with a
transverse representative that does not maximize self-linking number
but is non-destabilizable.

\begin{figure}
\centerline{
\includegraphics[width=5.5in]{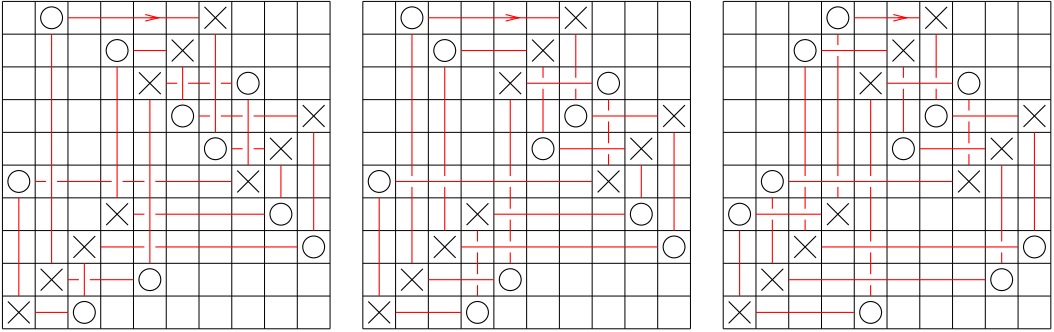}
}
\caption{
Grid diagrams representing nondestabilizable transverse knots of type
$m(10_{145})$, $m(10_{161})$, and $12n_{591}$.
}
\label{fig:nondestabex}
\end{figure}

\begin{proposition}
In each of the knot types $m(10_{145})$, $m(10_{161})$, and
$12n_{591}$, there are transverse knots
$T_1,T_2$ for which $sl(T_2) = sl(T_1)-2$ but $T_2$ is
not the stabilization of any transverse knot.
\label{prop:nondestabex}
\end{proposition}

\begin{proof}
This result can be proven using either of the transverse invariants
discussed in the previous section, but it is easiest to use the
$\widehat{\theta}$ invariant in knot Floer homology. Consider the grid
diagrams shown in Figure~\ref{fig:nondestabex}.\footnote{A note on
  conventions: to obtain grid diagrams as in
  Figure~\ref{fig:nondestabex} for which we can apply
  $\widehat{\theta}$ as in \cite{OST}, we either rotate a usual X-O
  diagram $90^\circ$ counterclockwise and interchange X's and O's (for
  the third diagram in Figure~\ref{fig:nondestabex}), or rotate a usual
  X-O diagram $90^\circ$ clockwise (for the first two diagrams). In
  the resulting diagrams, we use the convention from \cite{OST} that
  horizontal segments pass \textit{over} vertical segments. The
  resulting Legendrian front is either identical to the original front
  (for the third diagram), or related to the original front by the
  transformation $L \mapsto -\mu(L)$ (for the first two diagrams; for both,
  the atlas states that this transformation is a Legendrian
  isotopy), possibly along with a few elementary moves in Gridlink \cite{Cul}.
}
In each case, the
positive transverse pushoff $T_2$ of the grid diagram does not
maximize self-linking number, as can be seen by inspection of the atlas.
However, it can also be seen by inspection that $\widehat{\theta}$ is
nonzero for each of the diagrams: $\widehat{\theta}$ is represented in
the Manolescu--Ozsv\'ath--Sarkar complex for $\widehat{HFK}$ by
the upper-right corners of the X's, and it is clear in each case that
this generator is not in the image of the differential, since there
are no empty rectangles with NW-SE corners at two of these upper-right
corners. By a result of
\cite{OST}, $\widehat{\theta}=0$ for stabilizations of transverse
knots; it follows that $T_2$ is not a stabilization in each case.
\end{proof}

Two remarks are in order. First, the phenomenon of knots with a
non-maximal, non-destabilizable transverse representative was first
demonstrated by Etnyre and Honda \cite{EH2}, who showed that the
$(2,3)$ cable of the $(2,3)$ torus knot has this property; there is also
recent work by Lafontaine and Tosun, as well as Matsuda, in this
regard. However, the examples in Proposition~\ref{prop:nondestabex}
are significantly simpler in various ways than cables of torus
knots. For example, Shonkwiler and Vela-Vick \cite{SV} have shown that
the Legendrian contact homology of the $m(10_{161})$ knot in
Proposition~\ref{prop:nondestabex} is nontrivial, while an analagous
statement for the $(2,3)$ cable of the $(2,3)$ torus knot is still open.

Second, Proposition~\ref{prop:nondestabex} is an application of
the transverse HFK invariant that involves no computation, only an
inspection of a grid diagram. Previous applications of
$\widehat{\theta}$ to transverse nonsimplicity involved a computer
program (\cite{NOT}), an examination of naturality (\cite{OSt}), or
some relatively intricate linear algebra (\cite{KhN}). The simplicity
of the proof of Proposition~\ref{prop:nondestabex} suggests that the
knots considered there might be easily generalized to infinite
families of interesting nondestabilizable transverse knots. This is
indeed the case.

\begin{proposition}
For any $n \geq 1$, there is a topological knot $K_n$ with two
transverse representatives $T_{n,1},T_{n,2}$ such that
\[
sl(T_{n,2}) = sl(T_{n,1})-2n
\]
but $T_{n,2}$ is not the stabilization of any transverse knot. In
particular, $K_n$ is transversely nonsimple.
\label{prop:nondestabn}
\end{proposition}

\begin{figure}
\centerline{
\includegraphics[width=4.8in]{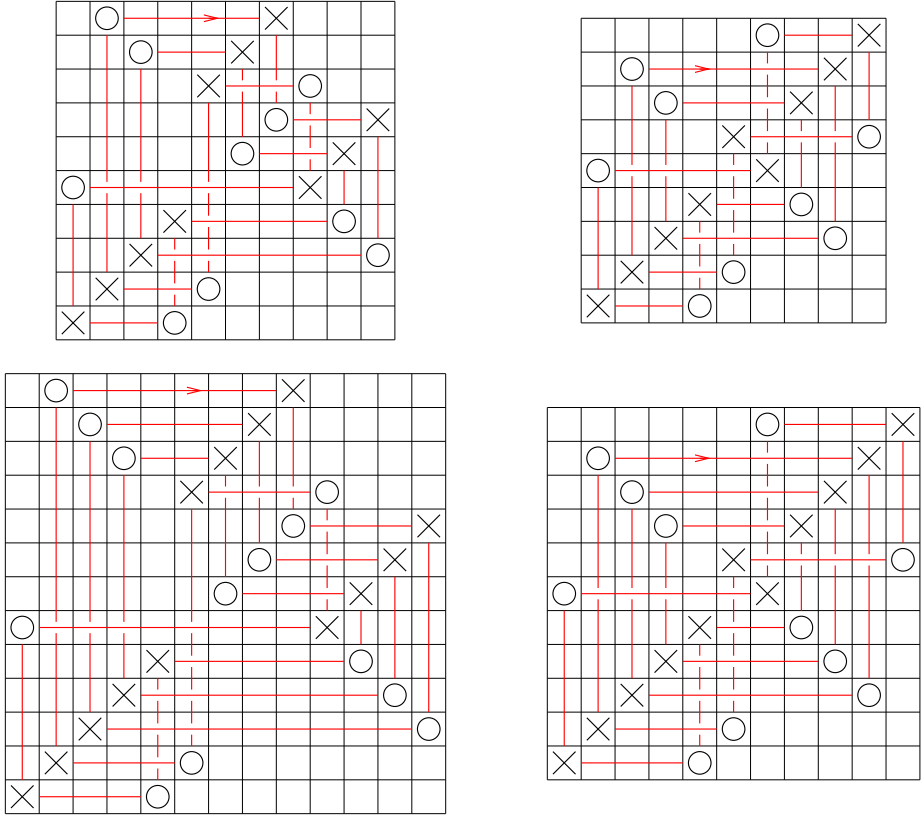}
}
\caption{
Grid diagrams for nondestabilizable transverse knots $T_{n,2}$ (left)
and topologically isotopic transverse knots $T_{n,1}$ 
(right), for $n=1,2$.
}
\label{fig:nondestab}
\end{figure}

Proposition~\ref{prop:nondestabn} as stated is
already known and follows from work of Etnyre and Honda
\cite{EHconnectsum}: given $K_1$, one can use the $n$-th
connected sum of $K_1$ with itself as $K_n$. However, the family $K_n$
we present in the proof of Proposition~\ref{prop:nondestabn} consists
of prime knots. We will only sketch a proof of primality and segregate
this result as Proposition~\ref{prop:prime} below.

It should be noted that Hiroshi Matsuda has independently 
obtained results similar
to Proposition~\ref{prop:nondestabn} (indeed, with apparently the same
family of examples); see the proof of Proposition~\ref{prop:prime}.

\begin{proof}[Proof of Proposition~\ref{prop:nondestabn}]
The grid diagram for $m(10_{161})$ given in
Figure~\ref{fig:nondestabex} generalizes readily to a family of grid
diagrams of size $3n+7$ and self-linking number $2n+1$, as shown (for
$n=1,2$) on the left hand side of Figure~\ref{fig:nondestab}; call the
positive transverse pushoff of these diagrams $T_{n,2}$.
By inspection, $\widehat{\theta}(T_{n,2}) \neq 0$ since there are no
empty rectangles with NW-SE corners at upper right corners of X's. On
the other hand, $T_{n,2}$ is evidently topologically isotopic to the
grid diagram on the right hand side of Figure~\ref{fig:nondestab},
whose positive transverse pushoff is a transverse knot $T_{n,1}$ with
self-linking number $4n+1$.
\end{proof}

\begin{proposition}
The family of knots $K_n$ in Proposition~\ref{prop:nondestabn} can be
chosen to be prime.
\label{prop:prime}
\end{proposition}

\begin{proof}[Outline of proof]
It is straightforward to check from Figure~\ref{fig:nondestab} that
$T_{n,2}$ is topologically the closure of the braid
\[
(\sigma_1 \sigma_2 \cdots \sigma_{n+2}) \sigma_{n+2} \sigma_{n+1}^{-3} \sigma_{n} \cdots \sigma_2 \sigma_1
(\sigma_1 \sigma_2 \cdots \sigma_{n+2}) \sigma_{n} \cdots \sigma_2 \sigma_1
(\sigma_1 \sigma_2 \cdots \sigma_{n+2}) \in B_{n+3};
\]
the diagram for $T_{n,2}$ is essentially braided clockwise around the
middle of the grid. We claim that topological type $K_n$ of $T_{n,2}$
is prime. The braid above is related by an operation discovered by
Matsuda called an
``H-flype'', which preserves topological knot type of the
braid closure, to the braid
\[
(\sigma_1\sigma_2)\sigma_2(\sigma_1\sigma_2)\sigma_2^{-3}
((\sigma_1\sigma_2)(\sigma_2\sigma_1))^{n} (\sigma_1\sigma_2)
= \sigma_2^{-2}\sigma_1(\sigma_2^2\sigma_1^2)^{n+1}\sigma_2
\in B_3.
\]
To show that the closure of a $3$-braid is prime, it suffices to check
that it is not a $(2,k)$ torus knot or the connected sum of two such
torus knots. In this case, this can be shown by calculating the
Alexander and Jones polynomials of the closure of
$\sigma_2^{-2}\sigma_1(\sigma_2^2\sigma_1^2)^{n+1}\sigma_2$ and
comparing to the Alexander and Jones polynomials of $(2,k)$ torus knots.
\end{proof}

\subsection{Non-maximal, non-destabilizable Legendrian knots}
\label{ssec:Legnondestab}

Each of the examples from the preceding section (non-destabilizable
transverse knots with non-maximal self-linking number) produces an
analogous phenomenon for Legendrian knots: a non-destabilizable
Legendrian knot with non-maximal Thurston--Bennequin number.

\begin{proposition}
In each of the knot types $m(10_{145})$, $m(10_{161})$, and
$12n_{591}$, there are Legendrian knots $L_1,L_2,L_3$
for which $tb(L_1) = tb(L_2)+1 = tb(L_3)+2$ but neither $L_2$ nor $L_3$ is
not the stabilization of any Legendrian knot.
\label{prop:nondestabexLeg}
\end{proposition}

\begin{proof}
We prove the result for $m(10_{145})$; the proof for the other two
knots is nearly identical. Let $L_1,L_2,L_3$ be the Legendrian $m(10_{145})$
knots in the atlas with $(tb,r)=(3,0),(2,1),(1,0)$, respectively; note that
$L_2$ is isotopic to the leftmost diagram in
Figure~\ref{fig:nondestabex}. Since the positive transverse pushoff of
$L_2$ is non-destabilizable by Proposition~\ref{prop:nondestabex},
$L_2$ is not the negative stabilization of any Legendrian knot. On the
other hand, the HOMFLY-PT polynomial bound of Morton and
Franks--Williams states for all Legendrian $m(10_{145})$ knots $L$
that $tb(L)+|r(L)| \leq 3$. In particular, there is no $L$ with
$(tb,r) = (3,2)$, and thus $L_2$ is not the positive stabilization of
any Legendrian knot.

The computer program of \cite{NOT} shows that $L_3$ and $-L_3$ both
have nonzero $\widehat{\theta}$ invariant; in the language of
\cite{NOT,OST}, both $\lambda_+(L_3)$ and $\lambda_-(L_3)$ are nonzero
in homology. Thus neither $L_3$ nor $-L_3$ is negatively
destabilizable, and it follows that $L_3$ is neither positively nor
negatively destabilizable.
\end{proof}

Shonkwiler and Vela-Vick \cite{SV} have provided an alternate proof
that the knot $L_2$ for $m(10_{161})$ in
Proposition~\ref{prop:nondestabexLeg} is non-destabilizable, using
Legendrian contact homology and the characteristic algebra.

One can extend the argument of Proposition~\ref{prop:nondestabn} to prove the
existence for any $n\geq 1$ of a prime knot $K_n$ with Legendrian
representatives
$L_{n,1},L_{n,2}$ for which $tb(L_{n,2}) = tb(L_{n,1})-2n$ but
$L_{n,2}$ is not destabilizable. We remark that B.\ Tosun has obtained
a similar result by studying cables of torus knots.

We can use the preceding discussion to examine the Legendrian mountain
range for $m(10_{145})$, with similar analysis for $m(10_{161})$ and
$12n_{591}$. The shape of the mountain range shown in the atlas is
determined by Proposition~\ref{prop:nondestabex}, along with the
following result.

\begin{proposition}
There are at least four distinct
$m(10_{145})$ knots with $(tb,r)=(1,0)$. (Note that $\overline{tb}(m(10_{145}))=3$.)
\end{proposition}

\begin{proof}
Let $L_1,L_2,L_3$ be the Legendrian $m(10_{145})$ knots from
(the proof of) Proposition~\ref{prop:nondestabexLeg}.
We claim that $S_+S_-(L_1)$, $S_-(L_2)$, $S_+(-L_2)$, and $L_3$
are pairwise distinct.

Since $L_3$ is non-destabilizable by the proof of
Proposition~\ref{prop:nondestabexLeg}, it is distinct from
$S_+S_-(L_1)$, $S_-(L_2)$, and $S_+(-L_2)$.
Since
$L_2$ and $S_+(L_1)$ have nonisotopic positive transverse pushoffs,
$S_-(L_2)$ and $S_-S_+(L_1)=S_+S_-(L_1)$ are distinct, as are
$S_+(-L_2) = -S_-(L_2)$ and $S_+S_-(L_1) = -S_+S_-(L_1)$.

It remains to show that $S_-(L_2)$ and $S_+(-L_2)$ are distinct.
One can verify by computer that $S_-^2(L_1)=S_-(-L_2)$, and so
$S_+S_-(-L_2) = S_+S_-^2(L_1)$. On the other hand, since $L_2$ and
$S_+(L_1)$ have distinct positive transverse pushoffs, $S_-^2(L_2)
\neq S_+S_-^2(L_1) = S_+S_-(-L_2)$. Thus $S_-(L_2) \neq S_+(-L_2)$, as desired.
\end{proof}

It is interesting to compare the mountain ranges of $m(10_{145})$,
$m(10_{161})$, and $12n_{591}$ with the mountain range of the $(2,3)$
cable of the $(2,3)$ torus knot from \cite{EH2}, which exhibits
similar behavior but is slightly different.

Besides $m(10_{145})$, $m(10_{161})$, and $12n_{591}$, the atlas
produces three other candidates for knots with non-maximal
non-destabilizable Legendrian representatives: $m(10_{139})$,
$10_{161}$, and $m(12n_{242})$. For these knots, it appears that a new
behavior emerges: the mountain ranges seem to have non-maximal peaks.

\begin{conjecture}
For each of $m(10_{139})$, $10_{161}$, and $m(12n_{242})$, there
exists a Legendrian knot $L$ with $tb(L)$ strictly less than the
maximal possible $tb$, for which there is no other Legendrian
representative with $(tb,r) = (tb(L)+1,r(L)+1)$ or $(tb(L)+1,r(L)-1)$.
\label{conj:nondestab}
\end{conjecture}

It may be worth remarking that the transverse techniques from the
previous sections are not applicable to
Conjecture~\ref{conj:nondestab}; it
appears that there is a unique non-destabilizable transverse knot in
each of the knot types. In addition, contact homology fails to provide an obstruction
to destabilizability: the non-maximal, conjecturally non-destabilizable
Legendrian knots of type $m(10_{139})$, $10_{161}$, and $m(12n_{242})$ in the atlas all
have vanishing Legendrian contact homology.

\begin{figure}
\centerline{
\includegraphics[width=2in]{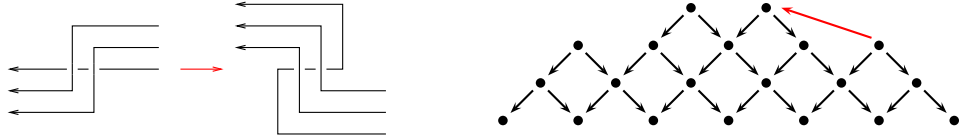} \hspace{0.5in}
\begin{pspicture}(6,1.9)
				\psdot(2.5,1.7)
				\psdot(3.5,1.7)
				\psline{->}(2.4,1.6)(2.1,1.3)
				\psline{->}(2.6,1.6)(2.9,1.3)
				\psline{->}(3.4,1.6)(3.1,1.3)
				\psline{->}(3.6,1.6)(3.9,1.3)
				\psdot(1,1.2)
				\psdot(2,1.2)
				\psdot(3,1.2)
				\psdot(4,1.2)
				\psdot(5,1.2)
				\psline{->}(0.9,1.1)(0.6,0.8)
				\psline{->}(1.1,1.1)(1.4,0.8)
				\psline{->}(1.9,1.1)(1.6,0.8)
				\psline{->}(2.1,1.1)(2.4,0.8)
				\psline{->}(2.9,1.1)(2.6,0.8)
				\psline{->}(3.1,1.1)(3.4,0.8)
				\psline{->}(3.9,1.1)(3.6,0.8)
				\psline{->}(4.1,1.1)(4.4,0.8)
				\psline{->}(4.9,1.1)(4.6,0.8)
				\psline{->}(5.1,1.1)(5.4,0.8)
				\psdot(0.5,0.7)
				\psdot(1.5,0.7)
				\psdot(2.5,0.7)
				\psdot(3.5,0.7)
				\psdot(4.5,0.7)
				\psdot(5.5,0.7)
				\psline{->}(0.4,0.6)(0.1,0.3)
	\psline{->}(0.6,0.6)(0.9,0.3)
				\psline{->}(1.4,0.6)(1.1,0.3)
				\psline{->}(1.6,0.6)(1.9,0.3)
				\psline{->}(2.4,0.6)(2.1,0.3)
				\psline{->}(2.6,0.6)(2.9,0.3)
				\psline{->}(3.4,0.6)(3.1,0.3)
				\psline{->}(3.6,0.6)(3.9,0.3)
				\psline{->}(4.4,0.6)(4.1,0.3)
				\psline{->}(4.6,0.6)(4.9,0.3)
				\psline{->}(5.4,0.6)(5.1,0.3)
				\psline{->}(5.6,0.6)(5.9,0.3)
				\psdot(0,0.2)
				\psdot(1,0.2)
				\psdot(2,0.2)
				\psdot(3,0.2)
				\psdot(4,0.2)
				\psdot(5,0.2)	
				\psdot(6,0.2)	
                                \psline[linecolor=red]{->}(4.9,1.3)(3.7,1.7)
		  \end{pspicture}
}
\caption{
A local move on Legendrian knots, preserving topological type and
changing $(tb,r)$ by $(+1,-3)$. On the right, the move in the context
of the mountain range for $m(10_{139})$, $10_{161}$, or
$m(12n_{242})$.
}
\label{fig:localmove}
\end{figure}

As a side note, each of the three knots in
Conjecture~\ref{conj:nondestab} has a ``local move'' relating a
non-maximal peak to a maximal peak, shown in
Figure~\ref{fig:localmove}. Presumably this move can be used to
construct many more examples of knots with non-maximal peaks in their
mountain ranges.

\subsection{Multiple linearized contact homologies}

Melvin and Shrestha \cite{MS} discovered the phenomenon of Legendrian knots that have more than one
possible linearized contact homology (corresponding to different augmentations of the Chekanov--Eliashberg
differential graded algebra). Their examples included the knots listed in our atlas as $m(8_{21})$ and (the second
representative of) $m(9_{45})$.

Our atlas provides more examples of Legendrian knots with multiple linearized contact homologies, of topological type
$11n_{95}$ and $11n_{118}$. In addition, the atlas finds another $m(9_{45})$ example with multiple linearized
contact homologies, distinct from the Melvin--Shrestha example.

\subsection{Discussion of other particular knots}

In the atlas, there are many instances of Legendrian knots with the
same classical invariants (topological type, Thurston--Bennequin
number, and rotation number) that are provably or conjecturally
non-isotopic. Often these can be distinguished from each other by
($0$-)graded ruling invariant or linearized contact homology, both included
in the atlas. The computations for ruling invariant and linearized
contact homology were performed using \cite{Sabloffprogram}.

Transverse nonsimplicity, as discussed in
Sections~\ref{ssec:transI} and~\ref{ssec:transII}, distinguishes
other Legendrian knots. For transversely nonsimple knots, there is a diagonal of
nonisotopic Legendrian knots with the same classical invariants that
travels down and to the left (following negative Legendrian
stabilization), and another diagonal down and to the right (following
positive stabilization).

Here we document the remaining cases of Legendrian knots in the atlas
that we can provably distinguish by other means. We use the convention
that in a particular knot type, the grid diagrams depicted in the
atlas represent Legendrian knots labeled $L_1,L_2,L_3,\ldots$ from top
to bottom.

\begin{itemize}

\item
$6_2$: The fact that the listed knot with $(tb,r) = (-7,0)$ is not
Legendrian isotopic to its mirror is proven in \cite{mir}, with a reproof in \cite{CLI}.

\item
$6_3$: The two Legendrian $6_3$ knots in the table have previously
been considered in \cite[section 4.3]{CLI}, where they are called
$K_4$ and $K_3$, respectively, and are proven to be distinct via the
characteristic algebra.

\item
$m(7_2)$: The fact that there are five distinct Legendrian representatives with $(tb,r)=(1,0)$, including a pair of nonisotopic mirrors, is proven in \cite{ENV} and essentially follows from the work on transverse twist knots in \cite{OSt}.

\item
$7_4$: Linearized contact homology distinguishes $L_4$ from the other
three. The knots $L_2,L_3$ were considered in \cite[section 4.2]{CLI},
where they were called $K_1,K_2$, respectively. As noted in
\cite{CLI}, these
two knots can be distinguished by their characteristic algebras. A minor
extension of the argument from \cite{CLI} also shows that $L_3=K_2$ is not
Legendrian isotopic to its Legendrian mirror: in the characteristic
algebra for $K_2$, there are elements $a_{13}$ and $a_{12}$ with
degrees $-2$ and $2$, respectively, for which $a_{13}a_{12}=1$, but no
elements $x,y$ with degrees $2$ and $-2$, respectively, for which
$xy=1$. See also \cite[section 4.1]{CLI}. Since the computer program
shows that $L_1$ and $L_2$ are each isotopic to their mirrors, neither
is isotopic to $L_3$. It is an interesting open problem to distinguish
$L_1$ and $L_2$.

\item
$m(7_6)$: Linearized contact homology distinguishes $L_3$ and $-L_3$
from $L_1$ and $L_2$. The computer program shows that $L_1,-L_3$ have
the same negative stabilization, as do $L_2,-L_2,L_3$; see Table~\ref{table:stab}. On the other hand, from \cite{Ngtranshom}, $L_1$ and $L_2$
represent distinct transverse knots. Thus $L_1,-L_3$ are distinct from
$L_2,-L_2,L_3$ as Legendrian knots. Since $L_1=-L_1$ by the computer
program, orientation reversal implies that $L_1,L_3$ are distinct from
$L_2,-L_2,-L_3$. As a result, $L_1,L_3,-L_3$ are pairwise distinct,
and all are distinct from $L_2,-L_2$. We do not currently know if
$L_2$ and $-L_2$ are isotopic.

\item
$9_{48}$, $m(10_{132})$, $10_{136}$, $m(10_{140})$: The Legendrian
knots of these types are distinguished using the data from Table~\ref{table:stab},
in a manner similar to $m(7_6)$ above.

\end{itemize}


\section{The Legendrian Knot Atlas}
\label{sec:atlas}

The table on the following pages depicts conjectural classifications
of Legendrian knots in all prime knot types of arc index up to
$9$. For each knot, we present a conjecturally complete list of
non-destabilizable Legendrian representatives, modulo the symmetries
of orientation reversal $L \mapsto -L$ and Legendrian mirroring $L
\mapsto \mu(L)$. As usual, rotate
$45^\circ$ counterclockwise to translate from grid diagrams to fronts.

Each knot also comes with its conjectural Legendrian mountain range
(extending infinitely downwards),
comprised of black and red dots, plotted according to their
Thurston--Bennequin number (vertical) and rotation number
(horizontal). Arrows represent positive and negative Legendrian
stabilization. The values of $(tb,r)$ are not labeled but can be
deduced from the values given for the non-destabilizable
representatives.  Boxes surround
values of $(tb,r)$ that have, or appear to have, more than one Legendrian
representative, and mountain ranges without boxes represent knot types
that are conjecturally Legendrian simple. The dots represent conjecturally
distinct Legendrian isotopy classes; black dots are provably distinct
classes, while red dots are conjecturally but not provably distinct from the black dots and each other. Thus the black
dots represent a lower bound for the Legendrian mountain range, and
the totality of dots represent our current best guess for the precise
mountain range (which however could theoretically be larger or smaller
than what is depicted).

Legendrian knots have been classified for several knot types,
including torus knots and $4_1$ \cite{EH} and twist knots
\cite{ENV}. These comprise the knots $3_1,4_1,5_1,5_2,6_1,7_1,7_2$ in
the table, along with their mirrors; for these knots, the mountain
ranges depicted in the atlas agree with the classification results. In
the table, we indicate torus knots by $T(p,q)$ and twist knots by
$K_n$ (for the knot with $n$ half-twists, with the convention of
\cite{ENV}).

Using symmetries, we can produce from any Legendrian knot $L$ up to
four possibly distinct Legendrian knots: $L$, $-L$, $\mu(L)$, and
$-\mu(L)$. The table depicts one representative from each of
these orbits of up to four knots, along with information about which
of the four knots in the orbit are isotopic, if any. For knots with
nonzero rotation number, we choose a representative $L$ with positive
rotation number, and $L$ is trivially distinct from $-L$ and $\mu(L)$
(this fact depicted by hyphens in the table).

Grid diagrams labeled with matching letters (see e.g. $6_2$) mark Legendrian knots that we believe but cannot yet prove to be distinct. Question marks indicate knots where we believe but cannot prove that $L$ is distinct
from $-L$, $\mu(L)$, or $-\mu(L)$. All check marks have been verified
by computer. All X marks without question marks have been proven, via
various techniques. These techniques include two nonclassical Legendrian invariants, the
graded ruling invariant \cite{ChP} and the set of (Poincar\'e polynomials for) linearized contact homologies \cite{Ch}, which have been computed, where relevant, using the
\textit{Mathematica} notebook \cite{Sabloffprogram}. (Knots
with no graded rulings/augmentations are denoted in these columns by a
hyphen, for nonzero rotation number, or $\emptyset$, for zero rotation
number.) For Legendrian knots that we have succeeded in distinguishing
by means besides these invariants,
please see Section~\ref{sec:phenomena} for documentation.

For some knots, the atlas omits a bit of information necessary to
deduce a complete (conjectural) Legendrian classification, namely
which Legendrian knots with the same $(tb,r)$ stabilize to isotopic
knots. This information is presented in Table~\ref{table:stab}, which
follows the atlas. The knots given in Table~\ref{table:stab} are those
where there is some ambiguity about isotopy classes after
stabilization; for all of those knots, the program guesses that the
relevant Legendrian representatives either become isotopic after one
(positive or negative) stabilization, or remain nonisotopic after
arbitrarily many stabilizations.


\input{atlas-source}

\begin{table}
\[
\hspace{-0.5in}
\begin{array}{|c||c|c|}
\hline
\text{Knot} & \text{Isotopy classes after $S_+$} & \text{Isotopy
  classes after $S_-$}\\
\hline\hline
m(7_2) & L_1,L_2,-L_3 \distinct L_3,L_4 & L_1,L_2,L_3 \distinct -L_3,L_4 \\
\hline
m(7_6) & L_1,L_3 \distinct L_2,-L_2,-L_3 &
L_1,-L_3 \distinct L_2,-L_2,L_3 \\
\hline
9_{44} & L_1,-\mu(L_1) \distinct L_2,-\mu(L_3) \maybe -\mu(L_2),L_3 &
-L_1,\mu(L_1) \distinct -L_2,\mu(L_3) \maybe \mu(L_2),-L_3\\
\hline
m(9_{45}) & L_1,\mu(L_1),\mu(L_2) \maybe -L_1,-\mu(L_1),L_2 &
L_1,\mu(L_1),L_2 \maybe -L_1,-\mu(L_1),\mu(L_2) \\
\hline
9_{48} & L_1,L_3 \distinct L_2,-L_2,-L_3,L_4,\mu(L_4) &
L_1,-L_3 \distinct L_2,-L_2,L_3,L_4,\mu(L_4) \\
\hline
10_{128} & L_1,-L_2 \maybe \mu(L_1),L_2 &
L_1,L_2 \maybe \mu(L_1),-L_2 \\
\hline
m(10_{132}) & L_1 \distinct -L_1,L_2 & L_1,L_2 \distinct -L_1 \\
\hline
10_{136} & L_1,L_4,\mu(L_4) \distinct -L_1,L_2,L_3,\mu(L_3) &
L_1,L_2,L_3,\mu(L_3) \distinct -L_1,L_4,\mu(L_4) \\
\hline
m(10_{140}) & L_1 \distinct -L_1,L_2 & L_1,L_2 \distinct -L_1 \\
\hline
m(10_{145}) & S_-(L_1) \distinct -L_2,L_3 &
S_+(L_1) \distinct L_2,L_3 \\
\hline
10_{160} & L_1,L_2,\mu(L_2) \maybe \mu(L_1),-L_2,-\mu(L_2) &
L_1,-L_2,-\mu(L_2) \maybe \mu(L_1),L_2,\mu(L_2) \\
\hline
m(10_{161}) & S_-(L_1) \distinct -L_2,L_3 &
S_+(L_1) \distinct L_2,L_3 \\
\hline
12n_{591} & S_-(L_1) \distinct -L_2,L_3 &
S_+(L_1) \distinct L_2,L_3 \\
\hline
\end{array}
\]
\caption{
Information about isotopy classes of Legendrian knots after
stabilization. For each knot type, $L_1,L_2,\ldots$ denote the
Legendrian knots depicted in the atlas, ordered from top to
bottom. In this table, knots separated by commas can be shown to be
Legendrian isotopic after one application of the appropriate
stabilization. Vertical bars separate knots that are provably distinct
after any number of the appropriate stabilizations; colons separate
knots that the program conjectures are distinct after any number of
stabilizations.
}
\label{table:stab}
\end{table}

\end{document}